\documentclass{amsart}

\usepackage[]{amsmath}
\usepackage[dvipsone]{graphics}
\usepackage{multicol}
\usepackage{caption}
\usepackage{comment}
\usepackage{color}
\usepackage{blkarray}
\usepackage{multirow}
\usepackage{amsthm}
\usepackage{ytableau}
\usepackage{fancyhdr}
\usepackage{tikz}
\usepackage{amsmath,amssymb, bbm}
\usepackage{color, colortbl}
\usepackage{amsthm}
\usepackage{graphicx}
\usepackage{subcaption}
\usepackage{mathdots}
\usepackage{etoolbox}
\let\bbordermatrix\bordermatrix
\patchcmd{\bbordermatrix}{8.75}{4.75}{}{}
\patchcmd{\bbordermatrix}{\left(}{\left[}{}{}
\patchcmd{\bbordermatrix}{\right)}{\right]}{}{}
\usetikzlibrary{decorations.pathreplacing}

\DeclareMathOperator{\Hess}{\mathit{Hess}}

\DeclareMathOperator{\cdelta}{\mathbb{C}[\alpha_i: \alpha_i\in \Delta ]}

    \title{The Containment Poset of Type $A$ Hessenberg Varieties}
    \author{Elizabeth Drellich}

\begin{document}
    \maketitle 
\begin{abstract}
Flag varieties are well-known algebraic varieties with many important geometric, combinatorial, and representation theoretic properties. A Hessenberg variety is a subvariety of a flag variety identified by two parameters: an element $X$ of the Lie algebra $\mathfrak{g}$ and a Hessenberg subspace $H\subseteq \mathfrak{g}$.  This paper considers when two Hessenberg spaces define the same Hessenberg variety when paired with $X$.  To answer this question we present the containment poset $\mathcal{P}_X$ of type $A$ Hessenberg varieties with a fixed first parameter $X$ and prove directly that if $X$ is not a multiple of the element $\bf 1$ then the Hessenberg spaces containing the Borel subalgebra determine distinct Hessenberg varieties. Lastly we give a natural involution on $\mathcal{P}_X$  that induces a homeomorphism of varieties and prove additional properties of $\mathcal{P}_X$ when $X$ is a regular nilpotent element.
\end{abstract}

\newtheorem{theorem}{Theorem}
\newtheorem{assertion}[theorem]{Assertion}
\newtheorem{Theorem}[theorem]{Theorem}
\newtheorem{proposition}[theorem]{Proposition}
\newtheorem{Proposition}[theorem]{Proposition}
\newtheorem{Lemma}[theorem]{Lemma}
\newtheorem{lemma}[theorem]{Lemma}
\newtheorem{Definition}[theorem]{Definition}
\newtheorem*{Definition*}{Definition}
\newtheorem{claim}[theorem]{Claim}
\newtheorem{Corollary}[theorem]{Corollary}
\newtheorem{condition}[theorem]{Condition}
\newtheorem{Conjecture}[theorem]{Conjecture}
\newtheorem{Example}[theorem]{Example}
\newtheorem{example}[theorem]{Example}
\newtheorem{Remark}[theorem]{Remark}
\newtheorem{remark}[theorem]{Remark}

\section{Introduction}

A Hessenberg variety is a subvariety of a flag variety $G/B$ determined by two parameters: $X$, an element of the underlying Lie algebra $\mathfrak{g}$, and a Hessenberg subspace $H \subseteq \mathfrak{g}$ (see Definition \ref{def: general hess space}
).  Any such pair of parameters defines a Hessenberg variety
$$\Hess(X,H) = \{ gB \in G/B : Ad(g^{-1})X \in H\}.$$

Hessenberg varieties appear in a number of mathematical fields under a number of names. For example Springer fibers, originally defined by Springer to construct representations of Weyl groups, can be expressed as $Spr(X) = \Hess(X,\mathfrak{b})$ \cite{Springer1978}. Peterson varieties first appeared in the construction of the quantum cohomology of the Grassmannian \cite{Peterson-notes}. Wachs and Shareshian's work on quasi-symmetric functions uses Hessenberg varieties to geometrically encode combinatorial data \cite{ShareshianWachs}.

Of course Hessenberg varieties are also interesting objects in their own right. People have studied the structure of certain subfamilies, like the Peterson varieties \cite{InskoYong}, as well as how Hessenberg varieties intersect with certain Schubert varieties \cite{InskoTymoczko}.  More generally the structure of Hessenberg varieties can be studied by paving them with affines \cite{ PrecupAffinePavings, TymozckoPureDimensional, TymoczkoPavingAffines}. Some Hessenberg varieties also carry non-trivial torus actions \cite{DeMariProcesiShayman92}, leading to the study of the torus equivariant cohomology of some Hessenberg varieties, such as those where $X$ is either a regular nilpotent element\cite{AbeHaradaHoriguchiMasuda, MyPeterson, HaradaTymoczkoPosetPinball} or a regular semisimple element of $\mathfrak{g} $ \cite{AbeHoriguchiMasuda}.

Despite their usefulness, it can be quite difficult to describe the structure of Hessenberg varieties directly. Even something as basic as determining when two Hessenberg varieties are equal can be non-trivial. 

\begin{Definition*}\cite[Def 1.6]{TymozckoPureDimensional}
\label{def: x-equiv}
Two Hessenberg spaces $H_1$ and $H_2$ are $X$-equivalent if $\Hess(X,H_1) = \Hess(X,H_2)$. 
\end{Definition*}
While a few $X$-equivalent  spaces are immediate, for example all Hessenberg spaces are $\bf 0$- and $\bf 1$-equivalent, non-trivial examples can be harder to identify. To determine where such examples exist, we define 
 the containment poset $\mathcal{P}_X$ of Hessenberg varieties with a fixed $X$.
 The containment poset on Schubert varieties is well known to be the {Bruhat} order on the corresponding Weyl group. The posets $\mathcal{P}_X$ vary widely depending on the choice of $X$ (see Examples \ref{ex: big poset proj},  \ref{ex: big poset nil}, and Figure \ref{fig: example px}). Nevertheless an interval at the top of the poset is the same for most choices of $X$.
Our main theorem (Theorem \ref{thm: new main thm}) shows that if $X\in \mathfrak{g}$ is not a multiple of $\bf 1$ then no two Hessenberg spaces containing $\mathfrak{b}$ are $X$-equivalent.  

 In Section \ref{sec: def} we present the necessary definitions and background on Hessenberg varieties. Section \ref{sec: poset} presents our motivating question on the existence of $X$-equivalent Hessenberg spaces and defines the poset $\mathcal{P}_X$ of Hessenberg varieties. We also present some of the diversity of structures $\mathcal{P}_X$ can have. The main theorem, that if $\mathfrak{b}\subseteq H$ and $X \neq a \bf 1$ then $\Hess(X,H)$ is uniquely determined by $H$, appears in Section \ref{sec: main thm} where we also give a constructive proof.  
 
 The last two sections of this paper present additional homeomorphisms between non-equal Hessenberg varieties and the implications of those homeomorphisms to the poset $\mathcal{P}_X$. Section \ref{sec: bilateral} proves that the bilateral symmetry of the containment poset of Hessenberg spaces induces a homeomorphism of the corresponding varieties. In Section \ref{Sec: Decomposable} we discuss which Hessenberg varieties are homeomorphic to products of other Hessenberg varieties and conclude with a condition  for when regular nilpotent Hessenberg varieties are decomposable in this sense. All of the  indecomposable regular nilpotent Hessenberg varieties are contained in a closed interval of $\mathcal{P}_X$ from the Peterson variety to the full flag variety.

The author would like to thank Julianna Tymoczko and Stephen Oloo for their helpful comments at multiple stages of this work, and Martha Precup for several inspiring and informative conversations.
\section{definitions}
\label{sec: def}
Throughout this paper we will refer to a complex reductive linear algebraic group $G$, a fixed Borel subgroup $B\subseteq G$, and their corresponding Lie algebras $\mathfrak{g}$ and $\mathfrak{b}$, and a root system $\Phi$. This construction also results in a flag variety $G/B$ and a Weyl group $W$.  Many of the proofs and examples in this paper rely on explicit constructions of individual full flags in the flag variety. A (full) flag $F_\bullet$ is sequence of nested subspaces in an ambient $n$ dimensional vector space $V$
$$\{0\} = F_0  \subsetneq F_1 \subsetneq \cdots \subsetneq F_n = V.$$
It will be convenient to think of $G, B$, and $F_\bullet$ as explicitly as possible. Throught this paper when discussing type $A_{n-1}$ Hessenberg varieties we will assume $G= Gl_n(\mathbb{C})$ and $B$ the upper triangular matrices in $G$. Furthermore a flag $F_\bullet$ will be represented as a matrix where the subspace $F_i$ is the span of the first $i$ columns.
\begin{Definition}
\label{def: Hess space}
A strict Hessenberg space is a subspace $H \subseteq \mathfrak{g}$ such that $\mathfrak{b} \subseteq H$ and $[H, \mathfrak{b}]\subseteq H$. 
\end{Definition}

This definition can be relaxed to allow for spaces that do not contain the Borel sub-Lie algebra and as we will see in Theorem \ref{thm: new main thm}, having $\mathfrak{b}$ contained in $H$ has significant implications for what Hessenberg varieties can occur.

There is a natural poset on the collection of  strict Hessenberg spaces ordered by inclusion. The unique maximal element of this poset is $\mathfrak{g}$, the unique minimal element is $\mathfrak{b}$.  Any strict Hessenberg space $H$ can be defined using a subset $\mathcal{M}_H$ of the negative roots in the corresponding root system, i.e.
$$
H = \mathfrak{b} \oplus \bigoplus \limits_{\alpha \in \mathcal{M}_H} \mathfrak{g}_\alpha.
$$
The space $H$ defined as above is a Hessenberg space if and only if  whenever $\alpha \in \mathcal{M}_H$ and another negative root $\beta $ is above $\alpha$ in the root lattice, $\beta$ is also in $\mathcal{M}_H$.

In type $A_{n-1}$ we use two additional ways to describe Hessenberg spaces: Hessenberg functions and matrix shape.  A strict Hessenberg function  $h:[n]\to [n]$ has $h(i)\geq \max\{i, h(i-1) \}$ for all $i$, and the root $-\alpha_i-\alpha_{i+1} -\cdots - \alpha_{j-1}$ is in $\mathcal{M}_H$ if and only if $h(i)\geq j$. As a space of matrices strict Hessenberg spaces have a staircase below the diagonal, below which all entries must be zero. A zero in the $(i,j)^{th}$ entry means that $-\alpha_i-\alpha_{i+1} -\cdots - \alpha_{j-1}$ is {\em not} in $\mathcal{M}_H$.
\begin{Example}{\rm
 \label{ex: Hessenberg spaces1}
There are five strict Hessenberg spaces in $\mathfrak{g}$ in type $A_2$.
{
$$\begin{matrix}&
{H}_1=\mathfrak{b}&
{H}_2&
{H}_3&
{H}_4&
{H}_5 =\mathfrak{g}\\
\\
&
\begin{bmatrix}
*&*&*\\
0&*&*\\
0&0&* \end{bmatrix}
&
\begin{bmatrix}
*&*&*\\
*&*&*\\
0&0&* \end{bmatrix}
&
\begin{bmatrix}
*&*&*\\
0&*&*\\
0&*&* \end{bmatrix}
&
\begin{bmatrix}
*&*&*\\
*&*&*\\
0&*&* \end{bmatrix}
&
\begin{bmatrix}
*&*&*\\
*&*&*\\
*&*&* \end{bmatrix}
\\
\\& \scriptstyle 1\mapsto 1 &\scriptstyle 1\mapsto 2 &\scriptstyle 1\mapsto 1 &\scriptstyle 1\mapsto 2 &\scriptstyle1\mapsto 3 \\
h& \scriptstyle 2\mapsto 2 &\scriptstyle 2\mapsto 2 &\scriptstyle 2\mapsto 3 &\scriptstyle 2\mapsto 3 &\scriptstyle 2\mapsto 3 \\
& \scriptstyle 3\mapsto 3 & \scriptstyle 3\mapsto 3 & \scriptstyle 3\mapsto 3 & \scriptstyle 3\mapsto 3 & \scriptstyle 3\mapsto 3 \\
\\
\mathcal{M}_{H} &
\emptyset &
 \{-\alpha_1\}&
  \{-\alpha_2\}&
  \{-\alpha_1, -\alpha_2\}&
\Phi^-
 \\
 \\
\end{matrix}$$}
}\end{Example}

The introduction of the Hessenberg function $h$ leads to a broader concept of a Hessenberg space.
\begin{Definition}
\label{def: general hess space}
 Given by an non-decreasing function $h: [n] \to [n]$, a (generalized) Hessenberg space is a subspace of $H \subseteq \mathfrak{g}$ given by 
 $$H = \bigoplus \limits_{h(i) \leq j }  \mathfrak{g}_{i,j}.$$
\end{Definition}

We will refer to the spaces defined in the previous two definitions as Hessenberg spaces, using the descriptors primarily to emphasize when strict Hessenberg spaces have special properties. Any one of these definitions of a Hessenberg space can be used to define a Hessenberg variety.
\begin{Definition}
\label{def: Hess var}
Fix a Hessenberg space $H$ and an operator $X \in \mathfrak{g}$. The Hessenberg variety $\Hess(X, H)$ is 
\begin{equation} 
\Hess(X,H) = \{ gB \in G/B : Ad(g^{-1})X \in H\}.
\end{equation}

\end{Definition}
The Hessenberg varieties form a very large family of algebraic varieties and many well-know varieties can be expressed as Hessenbergs. For example the full flag variety $G/B = \Hess(X,\mathfrak{g})$ for any $X\in \mathfrak{g}$ but if $X$ is a regular nilpotent element in $\mathfrak{g}$, then $\Hess(X,\mathfrak{b})$ is a single point.
\begin{Remark}
\label{rem: similar matrices}
An immediate consequence of the definition of a Hessenberg variety is that if $Y = P X P^{-1}$ is similar to $X$ then for any Hessenberg space $H$, $\Hess(X,H)$ is homeomorphic to to $\Hess(Y,H)$ \cite[Prop. 2.7]{LinearConditionsTymoczko}.
\end{Remark}

While the homeomorphism described above will be used in Theroem \ref{thm: flip}, unless otherwise specified all operators $X$ represent their entire similarity classes.  It will be useful for the purpose of explicit calculations to fix a particular regular nilpotent operator. 

\begin{Definition} 
\label{def: reg nil}
Fix a basis $\{E_\alpha \}$ for each subspace $\mathfrak{g}_\alpha \subseteq \mathfrak{g}$. Define a regular nilpotent operator $N$ as
\begin{equation}
N=\sum \limits_{\alpha \in \Delta} E_\alpha .
\end{equation}
\end{Definition}

In type $A_{n-1}$ using the standard basis $\{E_{(i,j)}\}$, the operator $N = \sum \limits_{i=1}^{n-1} E_{(i,i+1)}$ is the Jordan normal form of the nilpotent matrix with only one Jordan block.  

\begin{Example}
Let $G = GL_3(\mathbb{C})$, $B$ be the upper triangular matrices in $G$, and $H$ be the Hessenberg space defined by $h(1) = 2, h(2) = h(3) = 3$ ($H_4$ from Example \ref{ex: Hessenberg spaces1}).  There are number of ways to think about this variety:
\begin{itemize}
\item From the definition, $\Hess(N,H)= \{ gB \in G/B : Ad(g^{-1})N \in H\}. $
\item As flags, $\Hess(N,H) = \{F_\bullet \in G/B : NF_i \in F_{i+1}\}$.
\item Explicitly, these are the flags in the set $$\left\{ \begin{bmatrix} 1&0&0\\0&1&0\\ 0&0&1 \end{bmatrix},\begin{bmatrix} a&1&0\\1&0&0\\ 0&0&1 \end{bmatrix}, \begin{bmatrix} 1&0&0\\0&a&1\\ 0&1&0 \end{bmatrix},\begin{bmatrix} a&b&1\\b&1&0\\ 1&0&0 \end{bmatrix}\right\}.$$
\item Geometrically $\Hess(N,H)$ consists of two copies of $\mathbb{C}$ and one copy of $\mathbb{C}^2$  joined together at a single point.
\end{itemize}
\end{Example}
The variety $\Hess(N,H)$ in the previous example is the type $A_2$ Peterson variety.  Peterson varieties are the best studied regular nilpotent Hessenberg varieties and their singularities \cite{InskoYong} and cohomology \cite{MyPeterson} are understood. In general Lie type, the Peterson variety is the regular nilpotent Hessenberg variety corresponding to the Hessenberg space $$H = \mathfrak{b}  \oplus \bigoplus \limits_{\alpha \in -\Delta} \mathfrak{g}_\alpha.$$
In the sense that will be discussed in Section \ref{Sec: Decomposable} the Peterson variety is the smallest regular nilpotent Hessenberg variety in type $A$.

\section{Posets of Hessenberg Varieties}
\label{sec: poset}

Tymoczko gave the following definition for $X$-equivalence of Hessenberg spaces. 
\begin{Definition}\cite[Def 1.6]{TymozckoPureDimensional}
\label{def: x-equiv}
Two Hessenberg spaces $H_1$ and $H_2$ are called $X$-equivalent if $\Hess(X,H_1) = \Hess(X,H_2)$. 
\end{Definition}
There are two immediate occurrences of $X$-equivalence: all strict Hessenberg spaces are ${\bf 1}$-equivalent since $\Hess({\bf 1},H) = G/B$ for any strict Hessenberg space $H$. Similarly all Hessenberg spaces are $\bf 0$-equivalent. The search for non-trivial examples of $X$-equivalent strict Hessenberg spaces led the author to the main theorem of this paper, namely that no such examples exist.  However if we loosen our study to include all (generalized) Hessenberg spaces many non-trivial examples appear.
\begin{Example}
\label{ex: x-equiv}
Consider the operator $X$ and subspaces $H_1$ and $H_2$ defined by
$$X = \begin{bmatrix}
1&0&0&0\\
0&1&0&0\\
0&0&0&0\\
0&0&0&0\\
\end{bmatrix}, 
\hspace{5mm}
H_1= \begin{bmatrix}
0&*&*&*\\
0&0&*&*\\
0&0&*&*\\
0&0&*&*\\
\end{bmatrix}, 
\hspace{5mm}
H_2 = \begin{bmatrix}
0&0&*&*\\
0&0&*&*\\
0&0&*&*\\
0&0&*&*\\
\end{bmatrix}.$$

Consider a flag $F_\bullet \in \Hess(X,H_1)$. Using standard basis notation, $F_1$ must be in the span of $\{e_3, e_4\}$ since it must be in the null space of $X$. But since $XF_2$ must also be a subspace of the span of $\{e_3, e_4\}$, the subspace $F_2$ must itself be contained in that same span. There are no constraints on $F_3$ or $F_4$.  This flag $F_\bullet$ is, however, also contained in $\Hess(X,H_2)$ since $F_2$ is in the null space of $X$.  Therefore $\Hess(X,H_1) = \Hess(X,H_2)$.

\end{Example}

The collection of Hessenberg subspaces of any Lie algebra naturally form a poset ordered by containment. In type $A_{n-1}$ this poset is isomorphic to the closed interval in Young's lattice from $\emptyset$  to $n^n$ with the strict Hessenberg spaces making up the interval from $\emptyset$ to $(n, n-1, ... ,1)$. This can be seen using the notation in Example \ref{ex: Hessenberg spaces1} and reading each of the zeros as a box in the Young diagram using the French notation:

\begin{center}
\begin{tikzpicture}
\node[left] at (-1,0) {$\begin{bmatrix}
*&*&*\\
0&*&*\\
0&0&* \end{bmatrix}$}; 
\node[right] at (1,0){ \ydiagram[]{1,2}
}; 
\draw[<->, thick] (-.5, 0 )--(.5,0);
 \end{tikzpicture}
\end{center}
It will sometimes be notationally convenient to identify a Hessenberg space $H_\lambda$ by the corresponding Young diagram $\lambda$.
This identification leads immediately to an ordering on families of Hessenberg varieties.
\begin{Definition} For any $X\in \mathfrak{g}$ define the poset $P_X$ on the Hessenberg varieties $\Hess(X,H)$ by containment, i.e.  $\Hess(X,H_1)\preceq  \Hess(X,H_2)$ if the variety $\Hess(X,H_1) $ is contained in the variety $  \Hess(X,H_2)$.
\end{Definition}
This poset of Hessenberg varieties inherits some of the structure of the containment poset on Hessenberg spaces.
\begin{Lemma}
If $H_1 \subseteq H_2$ are Hessenberg spaces then $\Hess(X, H_1) \preceq \Hess(X,H_2)$ for every $X \in \mathfrak{g}$.
\end{Lemma}
\begin{proof}
 if $H_1 \subseteq H_2$ then $\Hess(X,H_1) \preceq  \Hess(X,H_2)$ since if $Ad(g^{-1})X \in H_1$ and $H_1 \subseteq H_2$ then $Ad(g^{-1})X \in H_2$. 
\end{proof}
So while it may be clear when one Hessenberg variety precedes another in $\mathcal{P}_X$ it is less clear when, or if, the poset $\mathcal{P}_X$ will be isomorphic to the interval $[ \emptyset, n^n]$ of Young's lattice. We demonstrate with two examples.

\begin{Example}
\label{ex: big poset proj}
Let $X = \begin{bmatrix} 1& 0 \\ 0& 0 \end{bmatrix}$. Let $eB$ and $s_1B$ be the elements in $GL_2(\mathbb{C})/B$ given by the identity matrix and the inversion $(1,2)$ respectively. Then the five (generalized) Hessenberg spaces give the varieties:
\begin{itemize}
\item $\Hess(X, H_\emptyset) = Gl_2(\mathbb{C})/B$
\item $\Hess(X,H_{(1)}) = \left \{ eB, s_1B  \right \}$
\item $\Hess(X,H_{(1,1)}) = \left \{ s_1B \right \}$
\item $\Hess(X,H_{(2)}) = \left \{ eB  \right \}$
\item $\Hess(X,H_{(2,1)}) = \Hess(X,H_{(2,2)}) = \emptyset$
\end{itemize}
\end{Example}
\begin{Example}
\label{ex: big poset nil}
Let $X = \begin{bmatrix} 0& 1 \\ 0& 0 \end{bmatrix}$. Again let $eB$  be the element in $GL_2(\mathbb{C})/B$ described above. Then the five (generalized) Hessenberg spaces give the varieties:
\begin{itemize}
\item $\Hess(X, H_\emptyset) = Gl_2(\mathbb{C})/B$
\item $\Hess(X,H_{(1)}) =\Hess(X,H_{(1,1)}) =  \Hess(X,H_{(2)}) = \Hess(X,H_{(2,1)}) = \left \{ eB  \right \}$
\item $ \Hess(X,H_{(2,2)}) = \emptyset$
\end{itemize}
\end{Example}

We visualize the posets $\mathcal{P}_X$ for these two examples in Figure \ref{fig: example px}. For aesthetic reasons we identify the variety $\Hess(X,H_\lambda)$ to the diagram $\lambda$.  While these examples are small, observe that no two varieties in the interval from $\emptyset$ to $\tiny \ydiagram[]{1}$ are equal. These are the varieties defined by strict Hessenberg spaces.
\begin{figure}[t]
\caption{Two posets $\mathcal{P}_X$. Here $\lambda$ denotes the Hessenberg variety $\Hess(X, H_\lambda)$.}
\label{fig: example px}
\begin{subfigure}[b]{0.3\textwidth}
\begin{center}
\begin{tikzpicture}
\node at (0,0) {$\emptyset$};
\node at (0,-1) {\tiny \ydiagram[]{1}};
\node at (-1,-2) {\tiny \ydiagram[]{1,1}};
\node at (1,-2) {\tiny \ydiagram[]{2}};
\node  at (0,-3.25) {\tiny \ydiagram[]{1,2} \, = \, \ydiagram[]{2,2}};
\draw [] (0,-.25)- - (0,-.75);
\draw [] (-.25, -1.25)-- (-.75,-1.75);
\draw [] (.25, -1.25)-- (.75,-1.75);
\draw [] (-.25, -2.75)-- (-.75,-2.25);
\draw [] (.25, -2.75)-- (.75,-2.25);
\end{tikzpicture}
\end{center}
\caption{$X = { \begin{bmatrix} 1&0\\0&0 \end{bmatrix}}$.}
\end{subfigure}
\hspace{.2 \textwidth}
\begin{subfigure}[b]{0.3\textwidth}
\begin{center}
\begin{tikzpicture}
\node at (0,0) {$\emptyset$};
\node at (0,-1.5) {\tiny \ydiagram[]{1,1} \, = \, \ydiagram[]{1}  \,=\, \tiny \ydiagram[]{2} \, = \, \ydiagram[]{1,2} };
\node  at (0,-3.35) {\ydiagram[]{2,2}};

\draw [] (0,-.25)- - (0,-.75);

\draw [] (0,-1.75)- - (0,-2.75);
\end{tikzpicture}
\end{center}
\caption{$X ={ \begin{bmatrix} 0&1\\0&0 \end{bmatrix}}$.}
\end{subfigure}

\end{figure}
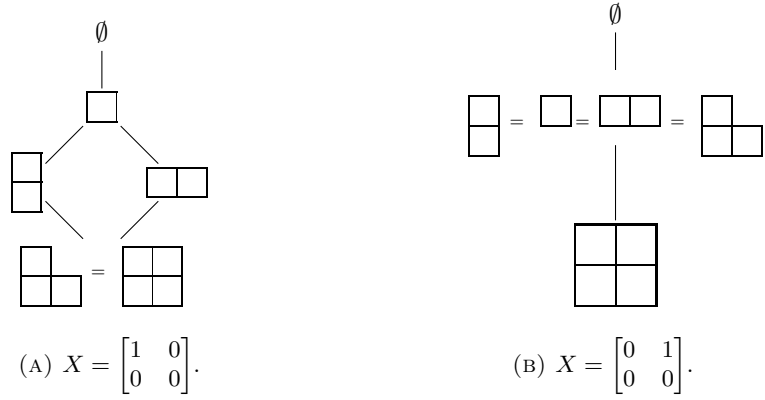

\section{Main Theorem}
\label{sec: main thm}
In this section we prove our main theorem, that no two strict Hessenberg spaces can be $X$-equivalent unless $X$ is a multiple of the identity element of $\mathfrak{g}$. 

\begin{Theorem}
\label{thm: new main thm}
Let $X\in \mathfrak{g}$ be an element that is not equal to $a  \bf 1$, and $H_1,H_2$ be distinct strict Hessenberg spaces. Then $\Hess(X,H_1) \neq \Hess(X,H_2)$.
\end{Theorem}

We will prove this directly by showing that for each negative root $\alpha \in \Phi^-$ we can explicitly construct a flag that is in $\Hess(X,H)$ if and only if $\alpha \in \mathcal{M}_H$.  Since we are in type $A$, root $\alpha = -\alpha_i - \alpha_{i+1} -\cdots -\alpha_j  \in \mathcal{M}_H$ is equivalent to $H$ having Hessenberg function with $h(i)\geq j$. Before constructing a general proof, we state a useful lemma and give an illustrative example.

\begin{Lemma}
\label{lem: identify a flag}
Let $X\in \mathfrak{g}$ be an operator, $i\leq j \leq n$, and $F_\bullet$ a flag in $\mathbb{C}^n$ with the following properties:
\begin{enumerate}
\item $XF_k \subseteq F_k$ whenever $k <i$ or $k >j$
\item $X F_i \subseteq F_j$
\item $X F_i $ is not a subspace of $ F_{j-1}$.
\end{enumerate}
Then the flag $F_\bullet$ is in $\Hess(X,H)$ if and only if the Hessenberg function for $H$ has $h(i) \geq j$.

\end{Lemma}

\begin{proof}
If $h(i) \geq j$ then the flag $F_\bullet$ as defined above is in $\Hess(X,H)$ by definition. If the flag $F_\bullet$ is in $\Hess(X,H)$ then because $XF_i$ is not a subspace of $F_{j-1}$, we must have $h(i) \geq j$. 
\end{proof}

The following example illustrates a flag that satisfies the properties of Lemma \ref{lem: identify a flag}.

\begin{example} Consider the operator
$$X = \begin{bmatrix} \lambda_1 &1 & & & & \\
& \lambda_1 &1 & & & \\
&&  \lambda_1 &  & & \\
  &&& \lambda_2 &1  & \\
    &&&&\lambda_2 &  \\
     &&&&& \lambda_3  \\
     \end{bmatrix}.$$
     
     To construct a flag which is contained in $\Hess(X, H)$ if and only if the Hessenberg function has $h(2) \geq 4$, we need a flag for which $XF_2$ is in $F_4$ but not $F_3$, and has $XF_i \subset F_i$ for $i = 1, 4,5,6$. One such example is 
     $$ A_{2,4} = \begin{bmatrix} e_4, e_2, e_5, e_1, e_6, e_3 \end{bmatrix} .$$
   Note that $e_1, e_4,$ and $e_6$ are eigenvectors of $X$ but $e_2$, which is contained in $F_2$, is not. The flag $A_{2,4}$ is in $\Hess(X,H)$ if and only if $h(2) \geq 4$.
     Using the same $X$,but $h(5) = 6$ we can build the flag
     $$ A_{5,6} = \begin{bmatrix} e_4, e_5, e_6, e_1, e_3, e_2 \end{bmatrix} $$ which is in $\Hess(X,H)$ if and only if $h(5) \geq 6$.
\end{example}

We now give a method for constructing such flags $A_{i,j}$ for any given operator $X \neq a {\bf 1}$. This construction proves the main theorem of this paper.

\begin{proof}[Proof of Theorem \ref{thm: new main thm}]
Without loss of generality we may assume that $X$ is an $n \times n$ matrix in Jordan normal form with blocks of size $\mu_1\geq \mu_2 \geq \cdots \mu_m \geq 1$.  Let $\{e_i\}_{i \leq n}$ be the standard basis of $\mathbb{C}^n$. Since $X$ is in Jordan normal form, $Xe_i \in {\rm span} \{e_i, e_{i-1}\}$ for all $i$. In order to construct our flags explicitly, we need to consider two separate cases: when $\mu_1 > 1$ and when $\mu_1 = 1$.

{\em The operator $X$ has a Jordan block of size $\mu_1 >1$.}
If the first Jordan block of $X$ has size $\mu_1>1$ then $X e_2 \not \in {\rm span} \{ e_2 \}$ but $Xe_2$ is in the span of $\{e_1, e_2\}$. Similarly for all $k = 2,3,...,\mu_1$, $X e_k $ is in the span of $\{e_{k-1}, e_k\}$ but not $\{ e_k \}$.
To build a flag $F_\bullet$  that has the properties required by Lemma \ref{lem: identify a flag}, first we reindex the basis vectors as $\{v_k\}_{ k \leq n}$ by
$$v_k = \begin{cases} 
e_{k+\mu_1}  &\text{ if } k\leq n-\mu_1\\
e_{k -n +\mu_1+3}  &\text{ if } k\geq n-\mu_1\\
e_2 & \text{ if } k = n-1 \\
  e_1 & \text{ if } k = n. \\
\end{cases}$$
This re-ordering has the property that unless $k = n-\mu_1$ or $n-1$, then $Xv_k $ is in the span of  $\{v_k, v_{k-1}\}$.
If $i\leq n-\mu_1$ define $F_\bullet$ by the matrix $A\in Gl_n(\mathbb{C})$ with $k^{th}$ column $A_k$ given by:
$$A_k = \begin{cases}
v_k & \text{ if } k<i \\
v_{n-1} = e_{2} & \text{ if } k=i \\
v_{k-1} & \text{ if } i< k <j\\
v_n = e_{1} & \text{ if } k=j \\
v_{k-2} & \text{ if }j<k.\\
\end{cases}$$
Since $i$ is less than $n-\mu_1$, we have that $X F_k \subset F_k$ unless $i \leq k \leq {j-1}$, in which case $XF_k$ is contained in $F_j$ but not $F_{j-1}$. \\
\\
If $i> n-\mu_1$ define $F_\bullet$ by the matrix $A\in Gl_n(\mathbb{C})$ with $k^{th}$ column $A_k$ given by:
$$A_k = \begin{cases}
v_k & \text{ if } k \leq n-\mu_1 \\
e_{k-n+\mu_1} & \text{ if } n- \mu_1< k <i \\
e_{k-n+\mu_1+1} & \text{ if } i \leq k =  <j \\
e_{{i-n+\mu_1}} & \text{ if } k=j \\
e_{k-n+\mu_1} & \text{ if } j < k .  \\
\end{cases}$$

Again the flag defined by $A$ has all the properties required by Lemma \ref{lem: identify a flag} and thus is contained in $\Hess(X,H)$ if and only if the Hessenberg function of $H$ has $h(i) \geq j$. \\
\\
{\em All Jordan blocks of $X$ are size $1$.}
If all Jordan blocks are size $1$ and $X$ is not a multiple of the identity matrix, then there are two standard basis vectors that have different eigenvalues $\lambda_1$ and $\lambda_2$. Without loss of generality we may assume that those are vectors $e_1$ and $e_2$. 

In a similar fashion to the case where $X$ has a Jordan block of size greater than one, we reorder the standard basis as follows and create a flag $A$ using a matrix in which the $k^{th}$ column $A_k$ is given by:

$$A_k = \begin{cases}
e_{k+2} & \text{ if } k<i \\
e_1+ e_{2} & \text{ if } k=i \\
e_{k+1} & \text{ if } i< k < j \\
e_1  & \text{ if } k = j \\
e_k  & \text{ if }  j < k. \\
\end{cases}$$
This flag is contained in $\Hess(X,H)$ if and only if $h(i) \geq j$, completing the proof.

\end{proof}

Unsurprisingly, the constructions above point to methods for constructing large families of flags that are in a variety $\Hess(X,H)$ if and only if $h(i) \geq j$, however the particular flags constructed have several advantages: if $X$ has a nilpotent part, then the flag $A$ is a permutation matrix. Furthermore if $X$ is actually a nilpotent operator then $A$ is a fixed point of the one dimensional torus action on $\Hess(X,H)$. 

\section{An Involution on $\mathcal{P}_X$}
\label{sec: bilateral}
The bilateral symmetry that the poset of type $A$ Hessenberg spaces inherits from Young's lattice leads to a natural question: does that bilateral symmetry have implications for the Hessenberg varieties? The answer is yes, the respective varieties are homeomorphic to each other. 

\begin{Proposition}
If $H$ is a Hessenberg space and ${}^TH\subseteq \mathfrak{g}$ is the set of matrices in $\mathfrak{g}$ obtained by flipping the matrix form of $H$ across its antidiagonal, then  ${}^TH$ is also a Hessenberg space. 
\end{Proposition}
\begin{proof}
If $H= H_\lambda$ for some Young diagram $\lambda$ then ${}^TH = H_{\lambda^T}$ which is also a Hessenberg space.
\end{proof}
Matrix manipulation shows that, when $w_0$ is the permutation matrix with ones on the antidiagonal, which corresponds to the longest word in $W$, any $n\times n$ matrix $M$ can be flipped along its antidiagonal by taking its transpose and conjugating by the longest word:
\begin{equation}
\label{eq: antidiagonal matrix flip}
{}^TM=w_0 M^T w_0
\end{equation}

and so for Hessenberg spaces:
\begin{equation}
\label{eq: h flip}
{}^TH = w_0 H^T w_0.
\end{equation}
The fact that both taking the transpose and conjugating by an element of the Weyl group are homeomorphisms of the Lie algebra will be useful in the proofs of this section's theorems.
\begin{theorem}
\label{thm: flip}
Let ${H_\lambda}\subseteq GL_n(\mathbb{C})$ be a Hessenberg space and let $ {}^T{H} = H_{\lambda^T}$ be the Hessenberg space obtained by flipping along the antidiagonal.  Then for any $X\in \mathfrak{g}$
$$\Hess(X,{H})\cong \Hess(X, {}^T{H}).$$
\end{theorem}
\begin{proof}
Consider the homeomorphism
\begin{equation}
\label{eq: isomorphism}
\begin{matrix}
\phi :& G/B & \to & G/B\\
& gB & \mapsto & w_0(g^T)^{-1} w_0 B.\\
\end{matrix}\end{equation}

Direct computation shows that $g^{-1}Xg \in H_\lambda$ if and only if $w_0 g^T X^T (g^T)^{-1} w_0$ is in $w_0 H^T w_0 = H_{\lambda^T}$ so the map $\phi$ restricts to a homeomorphism of Hessenberg varieties
\begin{equation}
\Hess(X, H_\lambda) \cong \Hess(w_0X^Tw_0, H_{\lambda^T}).
\end{equation}
But we know from Remark \ref{rem: similar matrices} and the fact that $X$ and $w_0 X^Tw_0$ are similar matrices (if $X$ was in Jordan normal form this is just a rearrangement of the blocks) that $\Hess(w_0X^Tw_0, H_{\lambda^T}) \cong \Hess(X, H_{\lambda^T})$. Therefore $\Hess(X, H_{\lambda^T}) \cong \Hess(X, H_{\lambda^T})$ for any $X, H_\lambda$ and the natural involution on $\mathcal{P}_X$ induces a homeomorphism of algebraic varieties.
 \end{proof}
These two varieties are unlikely to be equal if $\lambda \neq \lambda^T$ and $\Hess(X,H_\lambda)$ is non-empty.  In Example \ref{ex: big poset proj} we saw that while both $\Hess(X, H_{(2)})$ and $\Hess(X, H_{(1,1)})$ are a single point, they are different points in $G/B$.

\section{Decomposability: Regular Nilpotent Hessenberg Varieties}
\label{Sec: Decomposable}
The product of two type $A$ flag varieties can be expressed as a Hessenberg variety inside a larger flag variety. For example if $\mathcal{F}_1 = Gl_3(\mathbb{C})/B $ and $\mathcal{F}_2 = Gl_2(\mathbb{C})/B$ and $N$ is the regular nilpotent element given in Definition \ref{def: reg nil} then 
$$\mathcal{F}_1\times \mathcal{F}_2 \cong \Hess(N, H) \subset Gl_5(\mathbb{C})/B$$
where $H$ has Hessenberg function $h(1) = h(2)=h(3) = 3$ and $h(4)=h(5)=5$. In such a situation the structure of the Hessenberg variety is fully determined by the structures of $\mathcal{F}_1$ and $\mathcal{F}_2$.

\begin{Definition}
\label{def: decomposable}

Fix a flag variety $G/B$ and its corresponding Lie algebras $\mathfrak{g}$ and $\mathfrak{b}$.
A Hessenberg variety $\Hess(X,H)$ is decomposable if for some $\mathfrak{g}_1, \mathfrak{g}_2 \subsetneq \mathfrak{g}$  
\begin{itemize}
\item there exist $X_1\in \mathfrak{g}_1$ and  $X_2\in \mathfrak{g}_2$ 
\item there exist $H_1 \subseteq \mathfrak{g}_1$ and  $H_2\subseteq \mathfrak{g}_2$ 

\end{itemize}
such that $$\Hess(X, H) \cong \Hess(X_1,H_1) \times \Hess(X_2, H_2).$$
\end{Definition}
Note that in the example above where the two components into which the Hessenberg variety decomposes are in fact full flag varieties, $H_1 = \mathfrak{g}_1$ and $H_2 = \mathfrak{g}_2$. Definition \ref{def: decomposable} does not preclude the possibility that one of the components is a single point.
\begin{Theorem}[\cite{mythesis}]
\label{thm: decompose a}
Let $\Hess(N,H)$ be a type $A_{n-1}$ regular nilpotent Hessenberg variety. If for some $i<n$, $h(i)=i$ in the Hessenberg function of $H$ then $\Hess(N,H)$ is decomposable.
\end{Theorem}
\begin{proof}
Let $h:[n] \to [n]$ be a Hessenberg function with $h(j)=j$ for some $j<n$ and let $\Hess(N,H)$ be the corresponding regular nilpotent Hessenberg variety. We define two new type-$A$ Lie algebras and root systems by letting $G_1=GL_j(\mathbb{C})$ and $G_2=GL_{n-j}(\mathbb{C})$ and $\mathfrak{g}_1, \mathfrak{g}_2$ be their respective Lie algebras. For each, we define a Hessenberg function:
\begin{equation}
\begin{matrix}
h_1: &[j]& \to &[j]&\text{    and    }&h_2: &[n-j]& \to &[n-j]\\
& i& \mapsto &h(i)&\text{  }& & i& \mapsto &h(i+j)-j\\
\end{matrix}
\end{equation}
These determine two regular nilpotent Hessenberg varieties $\Hess(N_1,H_1)$ and $\Hess(N_2,H_2)$ where $N_1=N|_{\mathfrak{g}_1}$ and $N_2=N|_{\mathfrak{g}_2}$ are regular nilpotent operators in $\mathfrak{g}_1$ and $\mathfrak{g}_2$ respectively.  We will show that \begin{equation}
\Hess(N,H)\cong \Hess(N_1,H_1)\times \Hess(N_2,H_2).
\end{equation}
Define a map  
\begin{equation} \begin{matrix}\Hess(N_1,H_1)\times \Hess(N_2,H_2)&\to & \Hess(N,H)\\
(V_\bullet^{(1)},V_\bullet^{(2)})&\mapsto & V_\bullet^{(1)}\oplus V_\bullet^{(2)} \end{matrix}
\end{equation}
If $V_\bullet^{(1)}$ and $V_\bullet^{(2)}$ are flags in the two smaller Hessenberg varieties, then $V_\bullet$ is the flag in $\Hess(N,H)$ where
\begin{equation}
V_i= \begin{cases}
V_i^{(1)} &\text{ if }i\leq j\\
V_j^{(1)} \oplus V_{i-j}^{(2)}  &\text{ if }i> j
\end{cases}
\end{equation} In matrix notation $V_\bullet =\begin{bmatrix} V_\bullet^{(1)} & * \\ 0& V_\bullet^{(2)} \end{bmatrix}$. \\
\\ 
To see that $V_\bullet \in \Hess(N,H)$ we observe that 
\begin{equation}\begin{array}{rll}
NV_i= &N_1 V_i^{(1)} \subset V_{h_1(i)}^{(1)}=V_{h(i)} & \text{ if } i\leq j \\
NV_i \subseteq &V_j^{(1)} \oplus N_2 V^{(2)}_{i-j} \subset V_j^{(2)} \oplus V_{h_2(i-j)}^{(2)}=V_{h(i)} & \text{ if } i> j.\end{array} \end{equation}
As a direct sum of linear operators, this map is injective.  It remains to be shown that every flag $V_\bullet$ in $\Hess(N,H)$ has this form.\\
\\
Let $V_\bullet = V_1 \subsetneq V_2 \subsetneq \cdots \subsetneq V_n = \mathbb{C}^n$ be a flag in $\Hess(N,H)$.  Let $v\in V_j$ be a vector.  Without loss of generality say that $v=(v_1, v_2, \ldots, v_p,0, \ldots ,0)$ where $v_p$ is the last non-zero entry.  By the definition of $\Hess(N,H)$ the vector $Nv$ is also in $V_j$ as are all the vectors $N^kv$ for $k$ a non-negative integer.\\
\\
Since $v_p$ is non-zero, the vectors $N^kv$ are non-zero when $k$ is less than $p$.  The collection of vectors $\{N^kv: k=0,1,\ldots, p-1 \}$ is a linearly independent set in the space $V_j$. We know $V_j$  has dimension $j$ by definition of the flag.  Therefore $p$ is less than or equal to $j$.  This means any vector $v$ in $V_j$ must be in the span of the first $j$ basis elements.  We conclude that $V_j$ is equal to the span of the first $j$ basis elements and that the matrix form of $V_\bullet$ looks like
\begin{equation}
V_\bullet= \begin{bmatrix}
V_\bullet^{(a)} &*\\
0& V_\bullet^{(b)}
\end{bmatrix}.
\end{equation}
Here we define the two smaller flags to be $
V_i^{(a)}=V_i
$ and  $
V_i^{(b)}=V_{i+j}/V_j
$.  By the definition of $h_1$ the flag $V_\bullet^{(a)}$ is in $\Hess(N_1,H_1)$. For $V_\bullet^{(b)}$ we have that $N_2V_i^{(b)}=N_2 (V_{i+j}/V_j)$ which is equal to  $(NV_{i+j})/V_j$ as a subspace of $\mathbb{C}^{n-j}$. Similarly $V_{h_2(i)}^{(b)}$ is equal to $V_{h(i+j)}/V_j$ as a subspace of $\mathbb{C}^{n-j}$. Therefore for any $V_\bullet \in \Hess(N,H)$\begin{equation}
NV_{i+j} \subset V_{h(i+j)} \iff (NV_{i+j})/V_j \subset V_{h(i+j)}/V_j \iff N_2V_i^{(b)} \subset V_{h_2(i)}^{(b)}.
\end{equation}
Thus every flag $V_\bullet$ in $\Hess(N,H)$ is the product of a flag in $\Hess(N_1,H_1)$ and a flag in $\Hess(N_2,H_2)$. This process of decomposing the regular nilpotent Hessenberg variety into the product of smaller varieties can be repeated until each Hessenberg function preserves only the largest element of its domain.
\end{proof}
\noindent
The Peterson variety, which is the regular nilpotent Hessenberg variety corresponding to the Hessenberg function $h(i) = i+1$ for $i<n$,  is the smallest indecomposable regular nilpotent Hessenberg variety in type $A$.
\begin{Proposition}
The type $A$ regular nilpotent Hessenberg varieties that are indecomposable are contained in a closed interval in the poset $P_N$. The bottom element of this interval is the Peterson variety and the top element is the full flag variety. 
\end{Proposition}
\begin{proof}
This is an immediate consequence of Theorem \ref{thm: decompose a}.
\end{proof}

\bibliography{mybib}{}
\bibliographystyle{plain}
\end{document}